\documentclass[12pt,english,reqno]{amsart}
\usepackage{amsthm,amsmath,amssymb,amsfonts}
\usepackage{mathtools,booktabs,float,caption,cite,setspace}

\usepackage[
  pdfauthor={},
  pdfkeywords={},
  pdftitle={},
  pdfcreator={},
  pdfproducer={},
  linktocpage,colorlinks,bookmarksnumbered,linkcolor=blue,
  citecolor=blue,urlcolor=blue]{hyperref}

\theoremstyle{plain}
\newtheorem{theorem}{Theorem}

\newtheorem{corollary}[theorem]{Corollary}

\theoremstyle{definition}

\newtheorem{remark}[theorem]{Remark}

\numberwithin{equation}{section}
\numberwithin{theorem}{section}
\numberwithin{table}{section}

\newcommand{\ZZ}{\mathbb{Z}}
\newcommand{\QQ}{\mathbb{Q}}
\newcommand{\KK}{\mathcal{K}}
\newcommand{\KE}{\widehat{\mathcal{K}}}

\newcommand{\D}{\mathbf{D}}
\newcommand{\DD}{\mathbb{D}}
\newcommand{\DB}{\mathfrak{D}}
\newcommand{\DDiv}{\mathbb{D}^\top}
\newcommand{\DDivn}{\mathbb{D}^\bot}
\newcommand{\DDivcomp}{\mathbb{D}^{\top^{\scriptstyle\star}}}
\newcommand{\CN}{\mathcal{C}}
\newcommand{\CP}{\mathcal{C}\mspace{1mu}'}
\newcommand{\CPnum}{C\mspace{1mu}'}

\newcommand{\SF}{\mathbb{S}}
\newcommand{\SH}{\mathcal{S}}
\newcommand{\SHeven}{\mathcal{S}_{\text{\rm even}}}
\newcommand{\TS}{\mathcal{T}}

\newcommand{\GN}{\mathbf{G}}
\newcommand{\HN}{\mathbf{H}}
\newcommand{\ON}{\mathbf{O}}
\newcommand{\PN}{\mathbf{P}}

\DeclarePairedDelimiter{\floor}{\lfloor}{\rfloor}
\DeclareMathOperator{\denom}{denom}
\DeclareMathOperator{\lcm}{lcm}
\DeclareMathOperator{\rad}{rad}

\DeclareMathOperator{\pval}{v\mspace{-1.5mu}}

\newcommand{\mids}{\,\mid\,}
\newcommand{\nmids}{\,\nmid\,}

\newcommand{\iffq}{\; \iff \;}
\newcommand{\impliesq}{\; \implies \;}
\newcommand{\andq}{\quad \text{and} \quad}
\newcommand{\withq}{\quad \text{with} \quad}

\newcommand{\phrase}[1]{``#1"}
\newcommand{\set}[1]{\left\{#1\right\}}

\begin{document}

\title[On Carmichael numbers and Bernoulli polynomials]
{On Carmichael and polygonal numbers, Bernoulli polynomials,\\
and sums of base-$\boldsymbol{P}$ digits}
\author{Bernd C. Kellner and Jonathan Sondow}
\address{G\"ottingen, Germany}
\email{bk@bernoulli.org}
\address{New York, USA}
\email{jsondow@alumni.princeton.edu}

\begin{abstract}
We give a new characterization of the set $\mathcal{C}$ of Carmichael numbers
in the context of $p$-adic theory, independently of the classical results of
Korselt and Carmichael. The characterization originates from a surprising
link to the denominators of the Bernoulli polynomials via the
sum-of-base-$p$-digits function. More precisely, we show that such a
denominator obeys a triple-product identity, where one factor is connected
with a $p$-adically defined subset $\mathcal{S}$ of the squarefree integers
that contains $\mathcal{C}$. This leads to the definition of a new subset
$\mathcal{C}'$ of $\mathcal{C}$, called the ``primary Carmichael numbers''.
Subsequently, we establish that every Carmichael number equals an explicitly
determined polygonal number. Finally, the set $\mathcal{S}$ is covered by
modular subsets $\mathcal{S}_d$ ($d \geq 1$) that are related to the Kn\"odel
numbers, where $\mathcal{C} = \mathcal{S}_1$ is a special case.
\end{abstract}

\maketitle

%%%%%%%%%%%%%%%%%%%%%%%%%%%%%%%%%%%%%%%%%%%%%%%%%%%%%%%%%%%%%%%%%%%%%%%%%%%%%%%%
% Section
%%%%%%%%%%%%%%%%%%%%%%%%%%%%%%%%%%%%%%%%%%%%%%%%%%%%%%%%%%%%%%%%%%%%%%%%%%%%%%%%

\section{Introduction}
\label{sec:intro}

A composite positive integer $m$ is called a \emph{Carmichael number}
if the congruence
\begin{equation} \label{eq:Fermat-congr}
  a^{m-1} \equiv 1 \pmod{m}
\end{equation}
holds for all integers $a$ coprime to $m$
(see \cite[Sec.\,A13]{Guy:2004}, \cite[Chap.\,2, Sec.\,IX]{Ribenboim:2012}).
Clearly, if $m$ were a prime, then this congruence would be valid by
Fermat's little theorem.

Let \phrase{number} mean \phrase{positive integer} unless otherwise specified,
and let $p$ always denote a prime. A first result on Carmichael numbers is the
following criterion (for a proof, see \cite{Conrad:2016} or
\cite[p.\,134]{Crandall&Pomerance:2005}).

\begin{theorem}[Korselt's criterion \cite{Korselt:1899} (1899)]
\label{thm:Korselt}
A composite number $m$ is a Carmichael number if and only if $m$ is squarefree
and every prime divisor $p$ of $m$ satisfies \mbox{$p - 1 \mid m - 1$}.
\end{theorem}

Korselt did not give any examples of such numbers, while Carmichael succeeded
in determining the first ones, e.g.,
\[
  561 = 3 \cdot 11 \cdot 17, \quad 1105 = 5 \cdot 13 \cdot 17, \andq
  1729 = 7 \cdot 13 \cdot 19.
\]
Apparently unaware of Korselt's result, Carmichael showed the following properties.

\begin{theorem}[Carmichael \cite{Carmichael:1910,Carmichael:1912} (1910,1912)]
\label{thm:Carmichael}
Every Carmichael number $m$ is odd and squarefree and has at least three prime
factors. If $p$ and~$q$ are prime divisors of $m$, then
\[
  \textup{(i)} \quad p - 1 \mid m - 1, \quad
  \textup{(ii)} \quad p - 1 \mid \frac{m}{p} - 1, \quad
  \textup{(iii)} \quad p \nmid q - 1.
\]
\end{theorem}

An easy consequence of part~(ii) is that (see \cite{Conrad:2016})
\begin{equation} \label{eq:p-sqrt}
  p < \sqrt{m}.
\end{equation}

Denote the set of Carmichael numbers by
\[
  \CN = \set{561, 1105, 1729, 2465, 2821, 6601, 8911, 10585, 15841, \dotsc}.
\]

In $1994$ Alford, Granville, and Pomerance \cite{AGP:1994} proved that $\CN$ is
infinite, i.e., \emph{infinitely many Carmichael numbers exist}. More precisely,
they showed that if $C(x)$ denotes the number of Carmichael numbers less than $x$,
then $C(x) > x^{2/7}$ for sufficiently large $x$. This was improved by
Harman~\cite{Harman:2008} in $2008$ to
\[
  C(x) > x^{1/3} \quad \text{for all large\ } x.
\]

In the other direction, Erd\H{o}s \cite{Erdos:1956}  in 1956 improved a result
of Kn\"odel \cite{Knoedel:1953b} to show that
\[
  C(x) < x^{1-c\log\log\log x / \log\log x} \quad \text{for all large\ } x,
\]
where $c>0$ is a constant. For which estimate is closer to the true asymptotic
for $C(x)$, see Granville and Pomerance's discussion in \cite{Granville&Pomerance:2002}
(see also \cite[Chap.\,4, Sec.\,VIII]{Ribenboim:2012}).

The purpose of the present paper is to give a new characterization of the
Carmichael numbers in the context of $p$-adic theory, independently of the
results of Korselt and Carmichael in Theorems~\ref{thm:Korselt} and~\ref{thm:Carmichael}.
The characterization originates from a surprising link to the denominators of
the Bernoulli polynomials via the sum-of-base-$p$-digits function~$s_p$.

The link is introduced in Sections~\ref{sec:Carmichael} and~\ref{sec:Bernoulli}.
Section~\ref{sec:Carmichael} also introduces a $p$-adically defined set of
squarefree integers $\SH \supset \CN$, and the subset of
\phrase{primary Carmichael numbers} $\CP \subset \CN$.
Section~\ref{sec:polygonal} establishes that every Carmichael number equals
an explicitly determined polygonal number.

Subsequently, Sections~\ref{sec:proofs-1}, \ref{sec:proofs-2}, and~\ref{sec:proofs-3}
contain the postponed proofs of the results in Sections~\ref{sec:Carmichael},
\ref{sec:Bernoulli}, and~\ref{sec:polygonal}, respectively.

Finally, in Section~\ref{sec:modular} the set $\SH$ is covered by \emph{modular}
subsets $\SH_d$ for $d = 1, 2, 3, \dotsc$, providing a modular generalization
of $\CN = \SH_1$. It turns out that each $\SH_d$ is contained in a certain
superset $\KE_d$ of the so-called $d$-Kn\"odel numbers $\KK_d$.

%%%%%%%%%%%%%%%%%%%%%%%%%%%%%%%%%%%%%%%%%%%%%%%%%%%%%%%%%%%%%%%%%%%%%%%%%%%%%%%%
% Section
%%%%%%%%%%%%%%%%%%%%%%%%%%%%%%%%%%%%%%%%%%%%%%%%%%%%%%%%%%%%%%%%%%%%%%%%%%%%%%%%

\section{Carmichael numbers and squarefree integers}
\label{sec:Carmichael}

Define $\SF$ to be the set of squarefree integers greater than $1$:
\[
  \SF = \set{2,3,5,6,7,10,\dotsc}.
\]
Denoting by $s_p(n)$ the \emph{sum of the base-$p$ digits of~$n$},
we further define two subsets of $\SF$, namely,
\begin{align*}
  \SH &:= \set{m \in \SF \,:\, p \mid m \impliesq s_p(m) \geq p}\\
\shortintertext{and}
  \CP &:= \set{m \in \SF \,:\, p \mid m \impliesq s_p(m) = p}.
\end{align*}
Note that $\CP$ is a subset of $\SH$. One computes that
\begin{align*}
  \SH &= \set{231, 561, 1001, 1045, 1105, 1122, 1155, 1729, 2002, \dotsc}\\
\shortintertext{and}
  \CP &= \set{1729, 2821, 29341, 46657, 252601, 294409, 399001, \dotsc}.
\end{align*}

We will show that $\CP \subset \CN$ (see Theorem~\ref{thm:main}).
If $m \in \CP$, then $s_p(m) = p$ for all primes $p \mid m$, so we call
$m$ a \emph{primary} Carmichael number (hence the notation $\CP$,
meaning \phrase{$\CN$ prime}). The first one is $1729$, Ramanujan's famous
\phrase{taxicab} number, defined by him as \phrase{the smallest number
expressible as the sum of two [positive] cubes in two different ways}
(see \cite[p.\,12]{Hardy:1940}). The first primary Carmichael number not
congruent to $1$ modulo $4$ is
\[
  1152271 \equiv 3 \pmod{4},
\]
while the first element of $\CP$ with more than three prime factors is
\[
  10606681 = 31 \cdot 43 \cdot 73 \cdot 109.
\]

We can now state our first main results. The following one extends parts
of Theorem~\ref{thm:Carmichael} to a larger set.

\begin{theorem} \label{thm:main}
There are the strict inclusions
\[
  \CP \subset \CN \subset \SH \subset \SF.
\]
Moreover, for any $m \in \SH$ each prime factor $p$ satisfies the
property~\eqref{eq:p-sqrt} that $p < \sqrt{m}$. In particular, $m$ must
have at least three (respectively, four) prime factors, if $m$ is odd
(respectively, even).
\end{theorem}

Theorem~\ref{thm:main} leads to a new criterion for the Carmichael numbers.

\begin{theorem} \label{thm:criterion}
We have the characterization
\[
  \CN = \set{m \in \SH \,:\, p \mid m \impliesq s_p(m) \equiv 1 \pmod{p-1}}.
\]
In other words, an integer $m > 1$ is a Carmichael number if and only if $m$
is squarefree and each of its prime divisors $p$ satisfies both
\[
  s_p(m) \geq p \andq s_p(m) \equiv 1 \pmod{p-1}.
\]
From this characterization it follows directly that $m$ is odd and has at least
three prime factors, each less than $\sqrt{m}$.
\end{theorem}

Unlike the criterion of Korselt, that in Theorem~\ref{thm:criterion} does not
\emph{assume} compositeness. Indeed, all results of Theorems~\ref{thm:main}
and~\ref{thm:criterion} are deduced only from properties of the function $s_p$.
In this vein, we can even sharpen the consequence of Theorems~\ref{thm:main}
and~\ref{thm:criterion} that $p < \sqrt{m}$ if $p \mid m$.

\begin{theorem} \label{thm:estimate}
For certain subsets $\TS \subseteq \SH$, we have the sharp estimate
\[
  p \leq \alpha_\TS \, \sqrt{m} \quad (m \in \TS\!, \, p \mid m)
\]
with
\[
  \alpha_\TS = 1 \Big/ \sqrt{2 - \frac{1}{q}} =
  \left\{
    \begin{array}{lll}
      0.7237\dotsc \!, & q = 11,    & \text{if $\TS = \SH$},\\
      0.7177\dotsc \!, & q = 17,    & \text{if $\TS = \CN$},\\
      0.7071\dotsc \!, & q = 66337, & \text{if $\TS = \CP$},
    \end{array}
  \right.
\]
and
\[
  \alpha_\TS = 1 \Big/ \sqrt{3 - \frac{1}{q}} = 0.5789\dotsc, \;\;
    q = 61, \;\; \text{if $\TS = \SHeven$},
\]
where $\SHeven := \set{m \in \SH : m \text{ is even}}$.
\end{theorem}

Interestingly, to achieve the nontrivial bounds in Theorem~\ref{thm:estimate},
in each of the sets $\SH$, $\CN$, and~$\CP$ we find certain \emph{polygonal numbers},
as discussed in Section~\ref{sec:polygonal} and Table~\ref{tbl:poly-min}.

It is not obvious from its definition that the set $\SH$ is infinite. However,
that is an immediate corollary of Theorem~\ref{thm:main} and the existence of
infinitely many Carmichael numbers. An independent proof showing directly that
$\SH$ is infinite, without involving the set $\CN$, would certainly be of interest.

\begin{corollary}
The set\, $\SH$ is infinite.
\end{corollary}

If one could show that $\CP$ is infinite, this would give not only a new proof
of the infinitude of Carmichael numbers, but also another proof that $\SH$ is
infinite.

\newpage

Let $\CPnum(x)$ and $S(x)$ count the numbers of elements of $\CP$ and $\SH$ less
than $x$, respectively. Table~\ref{tbl:distrib} reports the slow but steady
increase in size of $\CPnum(x)$ compared to $C(x)$ and $S(x)$.

\begin{table}[H] \small
\begin{center}
\begin{tabular}{lrrr}
  \toprule
  \multicolumn{1}{c}{$x$} & \multicolumn{1}{c}{$\CPnum(x)$} &
  \multicolumn{1}{c}{$C(x)$} & \multicolumn{1}{c}{$S(x)$}\\
  \midrule
  $10^3$ & $0$ & $1$ & $2$\\
  $10^4$ & $2$ & $7$ & $57$\\
  $10^5$ & $4$ & $16$ & $636$\\
  $10^6$ & $9$ & $43$ & $7048$\\
  $10^7$ & $19$ & $105$ & $75150$\\
  $10^8$ & $51$ & $255$ & $801931$\\
  $10^9$ & $107$ & $646$ & $8350039$\\
  $10^{10}$ & $219$ & $1547$ & \hspace*{0.2em}$86361487$\\
  \bottomrule
\end{tabular}

\caption{Distributions of $\CPnum(x)$, $C(x)$, and $S(x)$.}
\label{tbl:distrib}
\end{center}
\end{table}

For the values of $C(10^n)$ up to $n = 16$ and $n = 21$, as well as a more
detailed analysis of their distribution, see \cite{Granville&Pomerance:2002} and
Pinch \cite{Pinch:2007}, respectively. The primary Carmichael numbers with more
than three prime factors seem to occur \emph{rarely}. Indeed, up to $10^{10}$
there are only five elements of $\CP$ with four (but not more) prime factors.

Granville and Pomerance \cite{Granville&Pomerance:2002} gave a precise conjecture
that Carmichael numbers with exactly three prime factors should satisfy
\[
  C_3(x) = O(x^{1/3}/\log^3 x).
\]
Heath-Brown \cite{Heath-Brown:2007} showed the upper bound
$C_3(x) = O(x^{7/20+\varepsilon})$ for any fixed $\varepsilon > 0$.

%%%%%%%%%%%%%%%%%%%%%%%%%%%%%%%%%%%%%%%%%%%%%%%%%%%%%%%%%%%%%%%%%%%%%%%%%%%%%%%%
% Section
%%%%%%%%%%%%%%%%%%%%%%%%%%%%%%%%%%%%%%%%%%%%%%%%%%%%%%%%%%%%%%%%%%%%%%%%%%%%%%%%

\section{Bernoulli numbers and polynomials}
\label{sec:Bernoulli}

The \emph{Bernoulli polynomials} are defined by the generating function
\begin{alignat*}{3}
  \frac{t e^{xt}}{e^t - 1} &= \sum_{n \geq 0} B_n(x) \frac{t^n}{n!}
    \quad &&(|t| < 2\pi)\\
\shortintertext{where}
  B_n(x) &= \sum_{k=0}^{n} \binom{n}{k} B_k \, x^{n-k} \quad &&(n \geq 0)
\end{alignat*}
and $B_k = B_k(0) \in \QQ$ is the $k$th \emph{Bernoulli number}.

For $n \geq 1$ denote by $\D_n, \DD_n$, and $\DB_n$ the denominators
(see \cite{Kellner&Sondow:2018})
\begin{align*}
  \D_n  &:= \denom(B_n) = 2,6,1,30,1,42,1,30,1,66,\dotsc,\\
  \DD_n &:= \denom \bigl( B_n(x)-B_n \bigr) =  1,1,2,1,6,2,6,3,10,2,\dotsc,\\
  \DB_n &:= \denom \bigl( B_n(x) \bigr) = 2,6,2,30,6,42,6,30,10,66,\dotsc.
\end{align*}

The denominators of the Bernoulli numbers are well known by the
\emph{von Staudt--Clausen theorem} of $1840$ (see \cite{Clausen:1840,Staudt:1840})
to be
\begin{equation} \label{eq:denom-D}
  \D_n =
    \left\{
      \begin{array}{cl}
        2, & \text{if $n = 1$},\\
        1, & \text{if $n \geq 3$ is odd},\\
        \prod\limits_{p-1 \mids n} p, & \text{if $n \geq 2$ is even}.
      \end{array}
    \right.
\end{equation}

The initial connection between the Carmichael numbers and the denominators of
the Bernoulli numbers and polynomials results from the known relations
\begin{equation} \label{eq:CN-D-DB}
  m \in \CN \impliesq m \mid \D_{m-1} \mid \DB_{m-1}.
\end{equation}
The first relation, $m \in \CN \Rightarrow m \mid \D_{m-1}$, actually holds as
an equivalence: \emph{An odd composite number $m$ is a Carmichael number if and
only if $m$ divides $\D_{m-1}$} (see Pomerance, Selfridge, and
Wagstaff~\cite[p.\,1006]{PSW:1980}). The equivalence follows easily from
Korselt's criterion and the von Staudt--Clausen theorem.
The second relation, $\D_{m-1} \mid \DB_{m-1}$, is easily seen, since even
\begin{equation} \label{eq:DB-lcm-1}
  \DB_n = \lcm( \DD_n, \D_n )
\end{equation}
holds for all $n \geq 1$ (cf.~\cite[Thm.\,4]{Kellner&Sondow:2017}).
Now the sum-of-base-$p$-digits function $s_p$ comes into play, as follows.

The authors \cite{Kellner:2017,Kellner&Sondow:2017,Kellner&Sondow:2018} have
recently shown that the denominators of the Bernoulli polynomials $B_n(x)-B_n$
(which have no constant term) are given by the remarkable formula
\begin{equation} \label{eq:DD-prod}
  \DD_n = \prod_{s_p(n) \, \geq \, p} p
\end{equation}
in which the product is finite since $s_p(n) = n$ if $p > n$.
Moreover, the following relation, supplementary to \eqref{eq:DB-lcm-1},
holds for $n \geq 1$ (see \cite{Kellner&Sondow:2018}):
\begin{equation} \label{eq:DB-lcm-2}
  \DB_n = \lcm \bigl( \DD_{n+1}, \rad(n+1) \bigr)
\end{equation}
where $\rad(n) := \prod_{p \mids n} p$.

In particular, $\D_n$, $\DD_n$, and $\DB_n$ are squarefree. Furthermore,
these denominators obey the following properties (see \cite{Kellner&Sondow:2018}):
\begin{alignat}{3}
  \DD_n &= \lcm \bigl( \DD_{n+1}, \rad(n+1) \bigr), \quad
    &&\text{if $n \geq 3$ is odd}, \label{eq:DD-succ}\\
  \DB_n &= \lcm \bigl( \DB_{n+1}, \rad(n+1) \bigr), \quad
    &&\text{if $n \geq 2$ is even}, \label{eq:DB-succ}
\end{alignat}
and (see \cite{Kellner:2017})
\begin{equation} \label{eq:DD-rad}
  \rad(n+1) \mid \DD_n, \quad \text{if $n+1$ is composite}.
\end{equation}

To substantiate the relationship between the Carmichael numbers and the
Bernoulli polynomials, we introduce for $n \geq 1$ the decomposition
\begin{equation} \label{eq:decomp}
  \DD_n = \DDiv_n \,\cdot\, \DDivn_n
\end{equation}
where
\begin{equation} \label{eq:DDiv-def}
  \DDiv_n  := \prod_{\substack{p \mids n\\ s_p(n) \, \geq \, p}} p \andq
  \DDivn_n := \prod_{\substack{p \nmids n\\ s_p(n) \, \geq \, p}} p.
\end{equation}
Additionally, we define the \emph{complementary} number to $\DDiv_n$ for
$n \geq 1$ as
\begin{align}
  \DDivcomp_n := &\,\, \prod_{\substack{p \mids n\\ s_p(n) \, < \, p}} p,
    \label{eq:DDivcomp-def}\\
\shortintertext{which satisfies the relation}
  \rad(n) = &\,\, \DDiv_n \,\cdot\, \DDivcomp_n. \label{eq:rad-prod}
\end{align}

As an application of these definitions, the next theorem gives a complete
description of the structure of the denominator $\DB_n$ of the Bernoulli
polynomial $B_n(x)$ in terms of a decomposition of $\DB_n$ into three factors.
(The result may be compared to the von Staudt--Clausen theorem in \eqref{eq:denom-D},
which describes the structure of the denominator of the Bernoulli number $B_n$.)
Furthermore, we obtain for all squarefree numbers~$m > 1$ a generalization
of~\eqref{eq:CN-D-DB}, when omitting its middle term $\D_{m-1}$.

\begin{theorem} \label{thm:triple}
For $n \geq 1$ the denominator $\DB_n$ of the Bernoulli polynomial $B_n(x)$
splits into the \textbf{triple product}
\[
  \DB_n = \DDivn_{n+1} \,\cdot\, \DDiv_{n+1} \,\cdot\, \DDivcomp_{n+1}.
\]
Moreover,
\[
  m \in \SF \iffq m \mid \DB_{m-1}.
\]
\end{theorem}

The interplay of the three factors of $\DB_n$ instantly yields the two relations
\begin{equation} \label{eq:DB-rel}
  \DB_n = \DDivn_{n+1} \,\cdot\, \rad(n+1) = \DD_{n+1} \,\cdot\, \DDivcomp_{n+1}.
\end{equation}
Explicit product formulas for $\DB_n$, in the contexts of \eqref{eq:DB-lcm-1}
and~\eqref{eq:DB-rel}, are given in \cite[Thm.\,4]{Kellner&Sondow:2017} and
\cite[Cor.\,1]{Kellner&Sondow:2018}, respectively.

We can now state our second main result. It establishes a fundamental
relationship between the Bernoulli polynomials and the Carmichael numbers,
since $\CN \subset \SH$ by Theorem~\ref{thm:main}.

\begin{theorem} \label{thm:main2}
The following claims are true:
\begin{enumerate}
\item The sequence $(\DDiv_n)_{n \geq 1}$ contains all elements of $\SH$.
      More precisely,
      \[
        \qquad m \in \SH \cup \set{1} \iffq m = \DDiv_m \iffq m \mid \DD_m.
      \]
\item If $m+1$ is composite, then $\rad(m+1) \mid \DDivn_m$.
\end{enumerate}
\end{theorem}

\begin{corollary} \label{cor:CN-DD}
The sequence $(\DDiv_n)_{n \geq 1}$ contains all the Carmichael numbers.
More precisely,
\[
  m \in \CN \impliesq m = \DDiv_m, \!\quad m \mid \DD_m, \!\quad
  m \mid \DDivn_{m-1}, \!\quad \!\text{and} \!\quad m \mid \DD_{m-1},
\]
but the converse does not hold.
\end{corollary}

\begin{remark}
The sequence $(\DDivn_n\mspace{-2mu})_{n \geq 1}$ very rarely intersects the
Carmichael numbers. Indeed, the only example below $10^6$ is
$\DDivn_{198} = 2465 \in \CN$.
\end{remark}

\begin{remark}
Comparing the definitions of the complementary numbers $\DDivcomp_n$ and the
set $\SH$, one immediately observes that the sequence $(\DDivcomp_n)_{n \geq 1}$
cannot contain any elements of $\SH$, and thus none of the Carmichael numbers.
Interestingly, it turns out that $(\DDivcomp_n)_{n \geq 1}$ is connected with
the quotients $\DD_n/\DD_{n+1}$ and $\DB_n/\DB_{n+1}$ (as introduced in
\cite{Kellner&Sondow:2018}), which are integral for odd and even indices~$n$
by~\eqref{eq:DD-succ} and~\eqref{eq:DB-succ}, respectively.
\end{remark}

\begin{remark}
It was actually the observation of the unexpected relationship
\[
  m \in \CN \impliesq m = \DDiv_m,
\]
as stated in Corollary~\ref{cor:CN-DD}, which led to the new characterization
of the Carmichael numbers via the sum-of-base-$p$-digits function $s_p$,
given in Theorem~\ref{thm:criterion}.
\end{remark}

%%%%%%%%%%%%%%%%%%%%%%%%%%%%%%%%%%%%%%%%%%%%%%%%%%%%%%%%%%%%%%%%%%%%%%%%%%%%%%%%
% Section
%%%%%%%%%%%%%%%%%%%%%%%%%%%%%%%%%%%%%%%%%%%%%%%%%%%%%%%%%%%%%%%%%%%%%%%%%%%%%%%%

\section{Polygonal numbers}
\label{sec:polygonal}

Surprisingly, the polygonal numbers (see \cite[Chap.\,XVIII]{Beiler:1966} and
\cite[pp.\,38--42]{Conway&Guy:1996}) are connected with the Carmichael numbers
and the set $\SH$.

Initially, we consider the following polygonal numbers for $n \geq 1$:
\[
  \PN_n = n(3n-1)/2, \quad \HN_n = n(2n-1), \quad \ON_n = n(3n-2),
\]
which are the $n$th \emph{pentagonal}, \emph{hexagonal}, and \emph{octagonal}
numbers, respectively. They satisfy an important property when $n=p$ is an
odd prime:
\[
  s_p(\HN_p) = s_p(\ON_p) = p \andq s_p(2\PN_p) = p+1.
\]

\newpage

To establish a connection between the set $\SH$ and the polygonal numbers,
we first introduce some definitions. Define $P(n)$ to be the
\emph{greatest prime factor} of $n$ if $n \geq 2$, and set \mbox{$P(1) := 1$}.
Also, denote the (double-shifted) $p$-adic value of $n$ by
\[
  \ell(n) := \floor*{\frac{n}{P(n)^2}} = \min_{p \mids n} \floor*{\frac{n}{p^2}}.
\]
We shall use the abbreviation $\ell = \ell(n)$ later on, if there is no
ambiguity in context. Finally, we need \emph{Legendre's formula}
(see \cite[Sec.\,5.3, p.\,241]{Robert:2000}), which gives the $p$-adic valuation
$\pval_p$ of a factorial by
\begin{equation} \label{eq:sp-fac}
  \pval_p( n! ) = \frac{n - s_p(n)}{p-1}.
\end{equation}

For simplicity, twice a polygonal number will be called a \emph{quasi} polygonal
number. The next theorem shows the special cases when $m \in \SH$ equals a
(quasi) polygonal number $\HN_p$, $\ON_p$, or $2\PN_p$ with $p = P(m)$,
the classification being determined by the parameter $\ell(m)$.

\begin{theorem} \label{thm:polygonal}
Let $m \in \SH$, and set $p = P(m)$ and $\ell = \ell(m)$.
Then the following statements hold:
\begin{enumerate}
\item We have $\ell \geq 1$.

\item There is the equivalence
\begin{align*}
  \ell = 1 &\iffq m = \HN_p \text{\ is a hexagonal number}.\\
\intertext{
\item There is the equivalence}
  \ell = 2 &\iffq m = \begin{cases}
    \ON_p, & \text{if $s_p(m) = p$},\\
    2 \PN_p, & \text{if $s_p(m) > p$}.
  \end{cases}
\end{align*}
In particular, for $m \in \CP$ we have $\ell = 2$ if and only if $m = \ON_p$
is an octagonal number.
\end{enumerate}
\end{theorem}

As needed later, Table~\ref{tbl:poly-min} reports the first occurrences of the
polygonal numbers $\HN_p$ and $\ON_p$ in each of the sets $\SH$, $\CN$, and
$\CP$, as well as the first occurrence of $2\PN_p$ in $\SH$. In contrast to the
relatively small values in Table~\ref{tbl:poly-min}, the exceptionally large
number $8801128801$, which is indeed the least hexagonal number in $\CP$,
could be found only by a computer search.

\begin{table}[H] \small
\begin{center}
\begin{tabular}{ccc@{\hspace*{-0.1em}}ccc}
  \toprule
  \multicolumn{1}{c}{set} & \multicolumn{1}{c}{$m$} &
  \multicolumn{1}{c}{factors} & \multicolumn{1}{c}{$p=P(m)$} &
  \multicolumn{1}{c}{$\ell(m)$} & \multicolumn{1}{c}{number}\\
  \midrule
  $\SH$  & $231$ & $3 \cdot 7 \cdot 11$ & $11$ & $1$ & $\HN_{11}$\\
  $\CN$  & $561$ & $3 \cdot 11 \cdot 17$ & $17$ & $1$ & $\HN_{17}$\\
  $\CP$ & $8801128801$ & $181 \cdot 733 \cdot 66337$ & $66337$ & $1$ & $\HN_{66337}$\\
  \midrule
  $\SH$  & $1045$ & $5 \cdot 11 \cdot 19$ & $19$ & $2$ & $\ON_{19}$\\
  $\CN$  & $2465$ & $5 \cdot 17 \cdot 29$ & $29$ & $2$ & $\ON_{29}$\\
  $\CP$ & $2821$ & $7 \cdot 13 \cdot 31$ & $31$ & $2$ & $\ON_{31}$\\
  \midrule
  $\SH$ & $11102$ & $2 \cdot 7 \cdot 13 \cdot 61$ & $61$ & $2$ & $2\PN_{61}$\\
  \bottomrule
\end{tabular}

\caption{\parbox[t]{20.5em}{First occurrences of (quasi) polygonal numbers
$\HN_p$ and $\ON_p$ in $\SH$, $\CN$, and $\CP$, as well as $2\PN_p$ in $\SH$.}}
\label{tbl:poly-min}
\end{center}
\end{table}

To generalize the results, we further consider the polygonal numbers of
rank~$r \geq 3$, also called \emph{$r$-gonal numbers}, namely,
\[
  \GN^r_n = \frac{1}{2}(n^2(r-2) - n(r-4)),
\]
where $\PN_n = \GN^5_n$, $\HN_n = \GN^6_n$, and $\ON_n = \GN^8_n$. Note though
that an \mbox{$r$-gonal} number can also be an $r'$-gonal number with
$r \neq r'$; for instance, $\GN^6_n = \GN^3_{2n-1}$ for $n \geq 1$.
Note also that $\GN^r_1 = 1$ and $\GN^r_2 = r$ cover for $r \geq 3$ all positive
integers except $2$. For that reason, our results on polygonal numbers $\GN^r_n$
will implicitly involve only those with $n \geq 3$. Clearly, for fixed $n \geq 3$
the sequence of such numbers $(\GN^r_n)_{r\geq 3}$ is strictly increasing.

To extend the results of Theorem~\ref{thm:polygonal}, one may ask which
$m \in \SH$ are equal to a (quasi) polygonal number $\GN^r_p$ or
$2\GN^r_p$ for $p = {P(m)}$ and some $r \geq 3$. The example
\[
  m = 2145 = \HN_{33} = 3 \cdot 5 \cdot 11 \cdot 13 \in \SH \setminus \CN
\]
shows that  $m$ is indeed a polygonal number, but $m$ is not of the form
$\GN^r_p$ with $p = P(m) = 13$, as verified by the consecutive values
\[
  \GN^{29}_{13} = 2119 < m = 2145 < 2197 = \GN^{30}_{13}.
\]

The next theorem clarifies this situation by showing that an element
$m \in \SH$ equals a (quasi) polygonal number $\GN^r_p$ or $2\GN^r_p$
with $p = P(m)$ if and only if $s_p(m)$ satisfies certain conditions.

\begin{theorem} \label{thm:polygonal2}
Let $m \in \SH$, and set $p = P(m)$ and $\ell = \ell(m)$.
Further define the integers $\eta \geq 1$ and $0 \leq \mu < p-1$ satisfying
\begin{equation} \label{eq:lambda-mu}
  s_p(m) = \eta(p-1)+\mu.
\end{equation}
The number $m$ equals a (quasi) polygonal number $\GN^r_p$ or $2\GN^r_p$
for some $r \geq 3$ if and only if $\mu$ can be written as
\begin{equation} \label{eq:cond-mu}
  \mu = d + e \, \frac{p-1}{2} \withq
  (d,e) \in \set{(1,0),(1,1),(2,0)}.
\end{equation}
Then in these cases we have that
\begin{equation} \label{eq:poly-S}
  m = d \cdot \GN^r_p \withq
  r = \frac{2}{d}(\ell + \pval_p(\ell !) + \eta + d) + e.
\end{equation}
\end{theorem}

As an application we obtain the following corollary for the Carmichael numbers.

\begin{corollary} \label{cor:CN-poly}
All Carmichael numbers are polygonal numbers.
More precisely, if $m \in \CN$, $p = P(m)$, and \mbox{$\ell = \ell(m)$}, then
\begin{equation} \label{eq:CN-poly}
  m = \GN^r_p \withq
  r = 2(\ell + \pval_p(\ell !) + \eta + 1),
\end{equation}
where $\eta \geq 1$ is the integer satisfying $s_p(m) = \eta(p-1)+1$.
In particular, relation~\eqref{eq:CN-poly} holds with $\eta = 1$
for all primary Carmichael numbers $m \in \CP$.
\end{corollary}

The first few numbers satisfying the conditions of Theorem~\ref{thm:polygonal}
are listed in Table~\ref{tbl:polygonal}. Additional numbers below $7000$
satisfying the conditions of Theorem~\ref{thm:polygonal2}, but not covered by
Theorem~\ref{thm:polygonal}, are listed in Table~\ref{tbl:polygonal2}.
As a special case, the taxicab number $1729$ equals the $12$--gonal number
$\GN^{12}_{19}$. Regarding Corollary~\ref{cor:CN-poly},
the first element of $\CN$ with $\eta = 2$ is
\[
  1050985 = 5 \cdot 13 \cdot 19 \cdot 23 \cdot 37 = \GN^{1580}_{37}.
\]

\begin{table}[H] \small
\begin{center}
\begin{tabular}{r@{$\;\in\;$}lcccc}
  \toprule
  \multicolumn{2}{c}{$m$} & \multicolumn{1}{c}{$p=P(m)$} &
  \multicolumn{1}{c}{$s_p(m)$} & \multicolumn{1}{c}{$\ell(m)$} &
  \multicolumn{1}{c}{number}\\
  \midrule
  $231$   & $\SH$ & $11$ & $11$ & $1$ & $\HN_{11}$\\
  $561$   & $\CN$ & $17$ & $17$ & $1$ & $\HN_{17}$\\
  $1045$  & $\SH$ & $19$ & $19$ & $2$ & $\ON_{19}$\\
  $2465$  & $\CN$ & $29$ & $29$ & $2$ & $\ON_{29}$\\
  $2821$  & $\CP$ & $31$ & $31$ & $2$ & $\ON_{31}$\\
  $3655$  & $\SH$ & $43$ & $43$ & $1$ & $\HN_{43}$\\
  $5565$  & $\SH$ & $53$ & $53$ & $1$ & $\HN_{53}$\\
  $8911$  & $\CN$ & $67$ & $67$ & $1$ & $\HN_{67}$\\
  $10585$ & $\CN$ & $73$ & $73$ & $1$ & $\HN_{73}$\\
  $11102$ & $\SH$ & $61$ & $62$ & $2$ & $2\PN_{61}$\\
  \bottomrule
\end{tabular}

\caption{\parbox[t]{16em}{The first (quasi) polygonal numbers
$\HN_p$, $\ON_p$, and $2\PN_p$ in $\SH$.}}
\label{tbl:polygonal}
\end{center}
\end{table}

\begin{table}[H] \small
\setstretch{1.25}
\begin{center}
\begin{tabular}{r@{$\;\in\;$}lcccc}
  \toprule
  \multicolumn{2}{c}{$m$} & \multicolumn{1}{c}{$p=P(m)$} &
  \multicolumn{1}{c}{$s_p(m)$} & \multicolumn{1}{c}{$\ell(m)$} &
  \multicolumn{1}{c}{number}\\
  \midrule
  $1105$ & $\CN$ & $17$ & $17$ & $3$  & $\GN^{10}_{17}$\\
  $1122$ & $\SH$ & $17$ & $18$ & $3$  & $2\HN_{17}$\\
  $1729$ & $\CP$ & $19$ & $19$ & $4$  & $\GN^{12}_{19}$\\
  $3458$ & $\SH$ & $19$ & $20$ & $9$  & $2\GN^{12}_{19}$\\
  $3570$ & $\SH$ & $17$ & $18$ & $12$ & $2\GN^{15}_{17}$\\
  $5005$ & $\SH$ & $13$ & $13$ & $29$ & $\GN^{66}_{13}$\\
  $5642$ & $\SH$ & $31$ & $32$ & $5$  & $2\ON_{31}$\\
  $6118$ & $\SH$ & $23$ & $24$ & $11$ & $2\GN^{14}_{23}$\\
  $6545$ & $\SH$ & $17$ & $17$ & $22$ & $\GN^{50}_{17}$\\
  $6601$ & $\CN$ & $41$ & $41$ & $3$  & $\GN^{10}_{41}$\\
  $6734$ & $\SH$ & $37$ & $38$ & $4$  & $2\GN^{7}_{37}$\\
  \bottomrule
\end{tabular}

\caption{\parbox[t]{17em}{Additional (quasi) polygonal numbers
$\GN^r_p$ and $2\GN^r_p$ in $\SH$ below $7000$.}}
\label{tbl:polygonal2}
\end{center}
\end{table}

%%%%%%%%%%%%%%%%%%%%%%%%%%%%%%%%%%%%%%%%%%%%%%%%%%%%%%%%%%%%%%%%%%%%%%%%%%%%%%%%
% Section
%%%%%%%%%%%%%%%%%%%%%%%%%%%%%%%%%%%%%%%%%%%%%%%%%%%%%%%%%%%%%%%%%%%%%%%%%%%%%%%%

\section{Proofs of Theorems \texorpdfstring{\ref{thm:main}, \ref{thm:criterion},
and \ref{thm:estimate}}{2.1, 2.2, and 2.3}}
\label{sec:proofs-1}

Recall the definitions and notation of Section~\ref{sec:polygonal}.
From Legendre's formula~\eqref{eq:sp-fac} one easily sees that
\begin{equation} \label{eq:sp-congr}
  n \equiv s_p(n) \pmod{p-1}.
\end{equation}

\begin{proof}[Proof of Theorem~\ref{thm:main}]
By the definitions and the computed examples,
we immediately obtain the strict inclusions $\CP \subset \SH \subset \SF$.

Given $m \in \SH$, we first show that $p \mid m$ implies $p < \sqrt{m}$.
As $m$ is squarefree, we can write
\begin{equation} \label{eq:padic-1}
  \frac{m}{p} = a_0 + a_1 \, p
\end{equation}
with $1 \leq a_0 \leq p-1$ and $a_1 \geq 0$. Since
\begin{equation} \label{eq:padic-2}
  a_0 + s_p(a_1) = s_p(m/p) = s_p(m) \geq p,
\end{equation}
we infer that $a_1 \geq 1$. Consequently, we obtain $a_0 + a_1 \, p > p$,
implying that $\sqrt{m} > p$. As a result, $m$ must have at least three
prime factors.

Now let $m$ be even. Suppose to the contrary that in this case $m$ has only
three prime factors. Hence we have
\begin{equation} \label{eq:m-even}
  m = 2 q p \withq p > q,
\end{equation}
where $p$ and $q$ are odd primes. By $s_p(2q) = s_p(m) \geq p$, we infer that
$2q \geq p$. Together with \eqref{eq:m-even} we then obtain that $2p > 2q > p$.
Using \eqref{eq:padic-1}, we conclude that $m / p = 2q = a_0 + a_1 \, p$ with
$a_1 = 1$. Since $a_0 = 2q - p < p-1$, it follows that $s_p(m) < p$, giving a
contradiction. Thus, if $m$ is even, then $m$ must have at least four prime
factors.

Next, we show that $\CN$ is equal to the set
\[
  \tilde{\SH} := \set{m \in \SH \,:\, p \mid m \impliesq s_p(m) \equiv 1 \pmod{p-1}}.
\]
Resolving the definition of $\SH$, for $m \in \tilde{\SH}$ we have the condition
\begin{equation} \label{eq:cond-1}
  p \mid m \impliesq s_p(m) \geq p \andq s_p(m) \equiv 1 \pmod{p-1}.
\end{equation}
Moreover, applying \eqref{eq:sp-congr} then yields
\begin{equation} \label{eq:congr-m-sp-1}
  m \equiv s_p(m) \equiv 1 \pmod{p-1}.
\end{equation}
For any $n \in \SF$, we have $n > 1$ is squarefree, so
\begin{equation} \label{eq:sp-1}
  s_p(n) = 1 \iffq n = p.
\end{equation}
By~\eqref{eq:congr-m-sp-1} and~\eqref{eq:sp-1}, condition~\eqref{eq:cond-1}
implies that
\begin{equation} \label{eq:cond-2}
  \text{$m$ is composite}, \andq p \mid m \impliesq p-1 \mid m-1.
\end{equation}
Thus $m$ satisfies Korselt's criterion. Hence, we conclude that
$\tilde{\SH} \subseteq \CN$.

Conversely, any $m \in \CN$ satisfies \eqref{eq:cond-2}. In view of
\eqref{eq:sp-congr}, we then have $s_p(m) \equiv 1 \pmod{p-1}$. Since $m$ is
squarefree and composite, from \eqref{eq:sp-1} we deduce that $s_p(m) \geq p$.
This implies that \eqref{eq:cond-1} holds, so $m \in \tilde{\SH}$ and
consequently $\CN \subseteq \tilde{\SH}$, proving that $\CN = \tilde{\SH}$.

Now, if $m \in \CP$, then \eqref{eq:cond-1} holds, so $m \in \CN$.
Considering the computed examples again, we finally deduce that
$\CP \subset \CN \subset \SH \subset \SF$.
This completes the proof of the theorem.
\end{proof}

\begin{proof}[Proof of Theorem~\ref{thm:criterion}]
The first statement is the equality $\CN = \tilde{\SH}$, established in the
proof of Theorem~\ref{thm:main}. Since $m \in \tilde{\SH}$ if and only if
\eqref{eq:cond-1} holds, the second statement then follows.

Moreover, Theorem~\ref{thm:main} also implies by $\CN \subset \SH$ that any
$m \in \CN$ has at least three prime factors, each satisfying $p < \sqrt{m}$.
As $m$ is composite and squarefree, an odd prime $p$ divides $m$.
Using \eqref{eq:sp-congr}, we then get relation \eqref{eq:congr-m-sp-1},
so $p-1 \mid m-1$, whence $m$ is odd.
\end{proof}

\begin{proof}[Proof of Theorem~\ref{thm:estimate}]
Consider a non-empty subset $\TS \subseteq \SH$ and define
\[
  \alpha_\TS := \sup_{m \, \in \, \TS} \frac{P(m)}{\sqrt{m}},
\]
where $\alpha_\TS \leq 1$ by Theorem~\ref{thm:main}.
Clearly, this definition includes for any $m \in \TS$ that
\begin{equation} \label{eq:estim-alpha}
  p \mid m \impliesq p \leq \alpha_\TS \, \sqrt{m},
\end{equation}
but it suffices to study the case where $p = P(m)$ is the greatest prime divisor
of $m$. To show that the estimate in \eqref{eq:estim-alpha} is sharp, we further
have to find an explicit $m' \in \TS$ such that $\alpha_\TS = P(m') / \sqrt{m'}$
holds.

Now, let $m \in \TS$. In view of \eqref{eq:padic-1} and \eqref{eq:padic-2},
we obtain by~\eqref{eq:estim-alpha} that
\begin{equation} \label{eq:estim-alpha2}
  \frac{1}{\alpha_\TS^2} \leq \frac{m}{P(m)^2} = \frac{a_0}{P(m)} + a_1
\end{equation}
with $1 \leq a_0 \leq P(m)-1$ and $a_1 \geq 1$. Thus, we are interested in
finding firstly a minimal number $a_1$, and secondly a minimal fraction
$a_0/P(m) \in (0,1)$. If they exist, then $\alpha_\TS$ is determined.

Next, we assume that there exists an element $m \in \TS$ with $a_1 = 1$.
(This is true for the sets of interest $\TS=\SH,\CN,\CP$.) From now on,
let $p = P(m)$. Since $m \in \SH$, we have the condition
$s_p(m) = a_0 + a_1 \geq p$, so $a_0 = p-1$.
Then \eqref{eq:estim-alpha2} becomes
\begin{equation} \label{eq:estim-alpha3}
  \frac{1}{\alpha_\TS^2} \leq \frac{m}{p^2} = \frac{p-1}{p} + 1 = 2 - \frac{1}{p}.
\end{equation}
Hence, to determine a minimal $\alpha_\TS$, we also have to determine a minimal~$p$
satisfying \eqref{eq:estim-alpha3}. Since $p = P(m)$ and $m = p(2p-1)$,
the factor~$p$ strictly increases with $m$. As a consequence, we can identify
the aforementioned element $m'$ as the minimal element $m' \in \TS$ for which
$a_1 = 1$. Finally, we achieve that
\[
  \alpha_\TS = 1 \Big/ \sqrt{2 - \frac{1}{P(m')}}.
\]

Now we use a link to the polygonal numbers. Since
\begin{equation} \label{eq:m-HN}
  m' = p(2p-1) = \HN_p,
\end{equation}
we have to find the least hexagonal number $\HN_p$ in each of the sets
$\TS = \SH, \CN, \CP$. This is done in Table~\ref{tbl:poly-min},
providing the solutions
\[
  P(m') = 11, 17, 66337 \quad \text{for} \quad \TS = \SH, \CN, \CP,
\]
respectively.

There remains the case when $m \in \SHeven$. For this purpose, let
$\TS = \SHeven$  with $m \in \TS$ and $p = P(m)$. Note that $p$ is odd,
since $m$ is composite. We adapt and reuse the arguments that lead
to~\eqref{eq:estim-alpha2} and~\eqref{eq:estim-alpha3}.
By~\eqref{eq:estim-alpha2} we have to find again a minimal $a_1 \geq 1$.
The case $a_1 = 1$ implies \eqref{eq:m-HN} and so an odd $m = \HN_p$ for odd~$p$.
Therefore, we show that case $a_1 = 2$ works, as follows.
By $s_p(m) = a_0 + a_1 \geq p$, we obtain two solutions $a_0 = p-2$ and
$a_0 = p-1$. Since $a_0 = p-2$ implies $m = p(3p-2) = \ON_p$, being odd for
odd $p$, there remains the case $a_0 = p-1$. Then we get $m = p(3p-1) = 2\PN_p$,
which is always even. Similar to~\eqref{eq:estim-alpha3}, we deduce that
\[
  \frac{1}{\alpha_\TS^2} \leq \frac{m}{p^2} = \frac{p-1}{p} + 2 = 3 - \frac{1}{p}.
\]

To find the minimal element $m' \in \TS$ with $a_1 = 2$, we have to find the
least quasi pentagonal number $2\PN_p$ in $\TS$. Table~\ref{tbl:poly-min} shows
that \mbox{$P(m') = 61$}. With that we finally obtain
\[
  \alpha_\TS = 1 \Big/ \sqrt{3 - \frac{1}{P(m')}}.
\]
This completes the proof of the theorem.
\end{proof}

%%%%%%%%%%%%%%%%%%%%%%%%%%%%%%%%%%%%%%%%%%%%%%%%%%%%%%%%%%%%%%%%%%%%%%%%%%%%%%%%
% Section
%%%%%%%%%%%%%%%%%%%%%%%%%%%%%%%%%%%%%%%%%%%%%%%%%%%%%%%%%%%%%%%%%%%%%%%%%%%%%%%%

\section{Proofs of Theorems \texorpdfstring{\ref{thm:triple} and \ref{thm:main2}}
{3.1 and 3.2} and Corollary \texorpdfstring{\ref{cor:CN-DD}}{3.3}}
\label{sec:proofs-2}

\begin{proof}[Proof of Theorem~\ref{thm:triple}]
From relations \eqref{eq:DB-lcm-2} and \eqref{eq:rad-prod} we get
\[
  \DB_n = \lcm \bigl( \DD_{n+1}, \DDiv_{n+1} \cdot \DDivcomp_{n+1} \bigr),
\]
and the decomposition~\eqref{eq:decomp} gives
$\DD_{n+1} = \DDiv_{n+1} \cdot \DDivn_{n+1}$.
Since $\DDiv_{n+1}$, $\DDivn_{n+1}$, and $\DDivcomp_{n+1}$ are pairwise coprime
by the definitions in \eqref{eq:DDiv-def} and \eqref{eq:DDivcomp-def},
the desired triple product formula follows.

If $m \in \SF$, then $m = \rad(m) >1$. By relation~\eqref{eq:DB-lcm-2},
we then have $m \mid \DB_{m-1}$.
Conversely, if $m \mid \DB_{m-1}$, then $m>1$ is squarefree, so $m \in \SF$.
This proves the required equivalence and completes the proof of the theorem.
\end{proof}

\begin{proof}[Proof of Theorem~\ref{thm:main2}]
We have to show two parts:

(i).~It suffices to prove the second statement. The definitions of $\SH$
and $\DDiv_n$ yield immediately that $m \in \SH \cup \set{1}$ if and only if
$\DDiv_m = m$. Since $\DDiv_n \mid \DD_n$ by \eqref{eq:decomp}, we have that
$\DDiv_m = m$ implies $m \mid \DD_m$. Conversely, if $m \mid \DD_m$, then
$\DDiv_m = m$ by~\eqref{eq:DD-prod} and~\eqref{eq:decomp}. This proves~(i).

(ii).~If $m+1$ is composite, then we have by~\eqref{eq:DD-rad}
and~\eqref{eq:decomp} that
\[
  \rad(m+1) \mid \DD_m = \DDiv_m \cdot \DDivn_m.
\]
Since $\gcd(m,m+1) = 1$, we infer by~\eqref{eq:DDiv-def} that
$\rad(m+1) \mid \DDivn_m$, proving~(ii).
\end{proof}

\begin{proof}[Proof of Corollary~\ref{cor:CN-DD}]
The implication follows from Theorem~\ref{thm:main2} parts (i) and~(ii),
using the strict inclusion $\CN \subset \SH$ and the compositeness of Carmichael
numbers. The converse does not hold, by Theorem~\ref{thm:main2} part~(i) and
considering $\SH \setminus \CN$.
\end{proof}

%%%%%%%%%%%%%%%%%%%%%%%%%%%%%%%%%%%%%%%%%%%%%%%%%%%%%%%%%%%%%%%%%%%%%%%%%%%%%%%%
% Section
%%%%%%%%%%%%%%%%%%%%%%%%%%%%%%%%%%%%%%%%%%%%%%%%%%%%%%%%%%%%%%%%%%%%%%%%%%%%%%%%

\section{Proofs of Theorems \texorpdfstring{\ref{thm:polygonal} and \ref{thm:polygonal2}}
{4.1 and 4.2} and Corollary \texorpdfstring{\ref{cor:CN-poly}}{4.3}}
\label{sec:proofs-3}

\begin{proof}[Proof of Theorem~\ref{thm:polygonal}]
Fix $m \in \SH$ and set $p = P(m)$ and $\ell = \ell(m)$.
We have to show three parts:

(i).~As $m$ is squarefree, we obtain
\begin{equation}\label{eq:padic-m-l}
  \frac{m}{p} = a_0 + a_1 \, p = a_0 + \ell \, p,
\end{equation}
where $1 \leq a_0 \leq p-1$ and $a_1 \geq 0$.
The case $a_1 = 0$ would imply $s_p(m) = s_p(m/p) = a_0 < p$.
Since $s_p(m) \geq p$ by $m \in \SH$, we must have $\ell = a_1 \geq 1$.
We shall use \eqref{eq:padic-m-l} implicitly in the remaining parts.

(ii).~If $\ell = 1$, then $s_p(m) = a_0+1 \geq p$, since $m \in \SH$.
But $a_0 \leq p-1$, so $a_0 = p-1$. Thus $m = p(p-1+p)= \HN_p$.
Conversely, if $m = \HN_p$, then $\ell = a_1 = 1$.

(iii).~Assume that $s_p(m) = p$. If $\ell = 2$, then $a_0+2 = s_p(m) = p$,
so $a_0 = p-2$ and $m = p(p-2+2p)= \ON_p$.
Conversely, if $m = \ON_p$, then $\ell = a_1 = 2$.

In particular, it then follows for $m \in \CP$ that $\ell = 2$ if and only if
$m = \ON_p$, since $s_p(m) = p$ by the definition of $\CP$.

Assume now that $s_p(m) > p$. If $\ell = 2$, then $a_0+2 = s_p(m) > p$,
so $a_0 > p-2$. But $a_0 \leq p-1$, so $a_0 = p-1$ and $m = p(p-1+2p) = 2 \PN_p$.
Conversely, if $m = 2 \PN_p$, then $\ell = a_1 = 2$.
This proves the theorem.
\end{proof}

\begin{proof}[Proof of Theorem~\ref{thm:polygonal2}]
Fix $m \in \SH$ and set $p = P(m)$ and $\ell = \ell(m)$. Since $s_p(m) \geq p$,
we can determine the integers $\eta \geq 1$ and $0 \leq \mu < p-1$
satisfying~\eqref{eq:lambda-mu}. Again, as in \eqref{eq:padic-m-l} we have
$m/p = a_0 + \ell \, p$, where $1 \leq a_0 \leq p-1$ and $\ell \geq 1$.
Using \eqref{eq:lambda-mu} we then obtain
\[
  \frac{m}{p} = \eta(p-1)+\mu - s_p(\ell) + \ell \, p.
\]
Now let $d \in \set{1,2}$. Resolving the desired equality
\begin{align*}
  d \cdot \GN^r_p &= m\\
\shortintertext{yields the equation}
  \frac{d}{2}(p^2(r-2) - p(r-4)) &= p \, (\eta(p-1)+\mu - s_p(\ell) + \ell \, p)
\end{align*}
with solution
\begin{equation} \label{eq:sol-r}
  r = \frac{2}{d} \left( \ell + \frac{\ell - s_p(\ell)}{p-1} + \eta + d + \frac{\mu - d}{p-1} \right).
\end{equation}
By Legendre's formula~\eqref{eq:sp-fac} we have
\begin{equation} \label{eq:pval-ell}
  \frac{\ell - s_p(\ell)}{p-1} = \pval_p(\ell!).
\end{equation}
Since $d \mid 2$ and $r \in \ZZ$, formulas \eqref{eq:sol-r}
and~\eqref{eq:pval-ell} imply the condition
\begin{equation} \label{eq:e-tilde}
  \tilde{e} := \frac{2}{d} \cdot \frac{\mu - d}{p-1} \in \ZZ.
\end{equation}

Since $m \in \SH$ has at least three prime factors by Theorem~\ref{thm:main},
we have $p = P(m) \geq 5$. This and the fact that $0 \leq \mu < p-1$ allow us
to continue deriving solutions of~\eqref{eq:e-tilde} for $\mu$, as follows.

In case $d = 2$, we infer that $\mu = 2$ and $\tilde{e} = 0$. In case $d=1$,
we get the solutions $\mu = 1$ and $\tilde{e} = 0$, as well as
$\mu = 1 + (p-1)/2$ and $\tilde{e} = 1$. One easily observes that all solutions
of~\eqref{eq:e-tilde} for $\mu$, $d$, and $\tilde{e}$ coincide with
condition~\eqref{eq:cond-mu} when taking $e = \tilde{e}$. Finally,
relation~\eqref{eq:poly-S} with $e = \tilde{e}$ follows from~\eqref{eq:sol-r}
by considering~\eqref{eq:pval-ell} and~\eqref{eq:e-tilde}.
This completes the proof of the theorem.
\end{proof}

\begin{proof}[Proof of Corollary~\ref{cor:CN-poly}]
If $m \in \CN$, then $s_p(m) = \eta(p-1)+1$ with $\eta \geq 1$ by
Theorem~\ref{thm:criterion}. In particular, if $m \in \CP$, then $\eta = 1$ by
definition of the set $\CP$. Since $\CN \subset \SH$ by Theorem~\ref{thm:main},
relation~\eqref{eq:CN-poly} follows by applying Theorem~\ref{thm:polygonal2}
with parameters $(d,e) = (1,0)$.
\end{proof}

%%%%%%%%%%%%%%%%%%%%%%%%%%%%%%%%%%%%%%%%%%%%%%%%%%%%%%%%%%%%%%%%%%%%%%%%%%%%%%%%
% Section
%%%%%%%%%%%%%%%%%%%%%%%%%%%%%%%%%%%%%%%%%%%%%%%%%%%%%%%%%%%%%%%%%%%%%%%%%%%%%%%%

\section{Modular properties of the set \texorpdfstring{$\SH$}{S}}
\label{sec:modular}

Define for a positive integer $d$ the subset $\SH_d$ of $\SH$ by
\[
  \SH_d := \set{m \in \SH \,:\, p \mid m \impliesq s_p(m) \equiv d \pmod{p-1}}.
\]
Theorem~\ref{thm:criterion} shows that $\SH_1 = \CN$. Thus the sets $\SH_d$ can
be viewed as a generalization, with the Carmichael numbers as a special case.
The first terms of the sets $S_d$ for $d = 1,2,3$ are
(compare Table~\ref{tbl:values})
\begin{align*}
  \SH_1 &= \set{561, 1105, 1729, 2465, 2821, 6601, 8911, 10585, 15841, \dotsc},\\
  \SH_2 &= \set{1122, 3458, 5642, 6734, 11102, 13202, 17390, 17822, \dotsc},\\
  \SH_3 &= \set{3003, 3315, 5187, 7395, 8463, 14763, 19803, 26733, \dotsc}.
\end{align*}

Let $\varphi$ denote \emph{Euler's totient function}.
The \emph{Carmichael function}~$\lambda$ (see \cite{Carmichael:1910})
is defined for $m = p_1^{e_1} \dotsm p_k^{e_k}$ with $p_1 < \dotsm < p_k$ by
\[
  \lambda(m) = \lcm( \lambda(p_1^{e_1}), \dotsc, \lambda(p_k^{e_k})),
\]
where $\lambda(p^e) = \delta \, \varphi(p^e)$ with $\delta = \frac{1}{2}$
if $p=2$ and $e \geq 3$, otherwise $\delta = 1$.

For positive integers $m$ the Carmichael function $\lambda$ has the property that
\begin{equation} \label{eq:lambda-congr}
  a^{\lambda(m)} \equiv 1 \pmod{m}
\end{equation}
holds for all integers $a$ coprime to $m$, where $\lambda(m)$ is the smallest
possible positive exponent. Since \eqref{eq:lambda-congr} generalizes the
Euler--Fermat congruence, it follows that $\lambda(m)$ divides $\varphi(m)$.
Moreover, for $m \in \SH$ we have the relation
\begin{equation} \label{eq:lambda-S-1}
  \lambda(m) = \lcm(p_1-1, \dotsc, p_k-1),
\end{equation}
where
\begin{equation} \label{eq:lambda-S-2}
  m = p_1 \dotsm p_k \geq 231 \quad \text{and} \quad k \geq 3.
\end{equation}

Define the \emph{function}~$\rho$ for positive integers $m$ by
\[
  \rho(1)=0, \quad \rho(2)=1,
\]
and
\begin{equation} \label{eq:rho-def}
  \rho(m) \equiv m \pmod{\lambda(m)} \quad (m \geq 3),
\end{equation}
being the least positive residue.

In view of \eqref{eq:lambda-congr} and \eqref{eq:rho-def}
the Fermat congruence~\eqref{eq:Fermat-congr} can be restated
for $m \geq 1$ in the form
\[
  a^{m-\rho(m)} \equiv 1 \pmod{m},
\]
holding for all integers $a$ coprime to $m$. As a special case, one has
\[
  \rho(m) = 1 \iffq \text{$m$ is prime} \text{ or } m \in \CN,
\]
which Carmichael proved with $m \equiv 1 \pmod {\lambda(m)}$ in place of
\mbox{$\rho(m) = 1$}.

Moreover, since $\lambda(m)$ is even for $m \geq 3$ by construction,
we have the parity relation
\begin{equation} \label{eq:rho-par}
  \rho(m) \equiv m \pmod{2} \quad (m \geq 3).
\end{equation}

\begin{table}[H] \small
\begin{center}
\begin{tabular}{c|*{9}{r}}
  \toprule
  $m$ & $231$ & $561$ & $1001$ & $1045$ & $1105$ & $1122$ & $1155$ & $1729$ & $2002$\\
  $\rho(m)$ & $21$ & $1$ & $41$ & $145$ & $1$ & $2$ & $15$ & $1$ & $22$\\
  $\lambda(m)$ & $30$ & $80$ & $60$ & $180$ & $48$ & $80$ & $60$ & $36$ & $60$\\
  \bottomrule
\end{tabular}

\caption{First values of $\rho(m)$ and $\lambda(m)$ for $m \in \SH$.}
\label{tbl:values}
\end{center}
\end{table}
\vspace*{-4ex}

\begin{theorem} \label{thm:rho}
If $m \in \SH$, then $\rho(m)$ equals the least positive index
$d < \lambda(m)$ such that $m \in \SH_d$. Moreover, we have
\[
  m \in \SH_{d \,+\, j \, \lambda(m)} \quad (j \in \ZZ_{\geq 0}).
\]
\end{theorem}

\begin{proof}
Given $m \in \SH$, factor $m = p_1 \dotsm p_k$ and consider by
\eqref{eq:lambda-S-2} and~\eqref{eq:rho-def} the congruences
\[
  d \equiv m \equiv \rho(m) \pmod{\lambda(m)}.
\]
From \eqref{eq:sp-congr} and \eqref{eq:lambda-S-1}, we further deduce the
system of congruences
\[
  d \equiv m \equiv s_{p_\nu}(m) \pmod{p_\nu-1} \quad (\nu = 1, \dotsc, k).
\]
Thus, $d = \rho(m) < \lambda(m)$ is the least positive index such that $m \in \SH_d$.
Moreover, it also follows that $m \in \SH_{d \,+\, j \, \lambda(m)}$ for $j \geq 1$.
\end{proof}

Define the \emph{$d$-Kn\"odel numbers} $\KK_d$ (see \cite{Knoedel:1953}) to be
the set of composite integers $m > d$ such that
\begin{equation} \label{eq:Kd-def}
  a^{m-d} \equiv 1 \pmod{m}
\end{equation}
holds for all integers $a$ coprime to $m$. (Note that the usual but equivalent
definition is further restricted to $1 < a < m$.) For example,
the \mbox{$1$-Kn\"odel} numbers are the Carmichael numbers: $\KK_1 = \CN$.
For $d=2,3$ the $d$-Kn\"odel numbers are
\begin{align*}
  \KK_2 &= \set{4, 6, 8, 10, 12, 14, 22, 24, 26, 30, 34, 38, 46, 56, 58, 62, 74, \dotsc},\\
  \KK_3 &= \set{9, 15, 21, 33, 39, 51, 57, 63, 69, 87, 93, 111, 123, 129, 141, \dotsc}.
\end{align*}

Makowski \cite{Makowski:1962} showed that each of the sets $\KK_d$ for
$d \geq 2$ is infinite. More precisely, for given $d \geq 2$ he proved the
existence of infinitely many primes $p > d$ such that
(see \cite[pp.\,125--126]{Ribenboim:2012})
\begin{equation} \label{eq:Kd-elem}
  dp \in \KK_d.
\end{equation}

Our final theorem shows properties of the sets $\SH_d$, as well as a connection
with generalizations of the sets $\KK_d$. Avoiding the restriction $m > d$ on
numbers $m \in \KK_d$, we define the \emph{superset} $\KE_d$ of $\KK_d$ to be
all composites $m > 1$ satisfying \eqref{eq:Kd-def} for all $a$ coprime to $m$.
Note that $\KK_1 = \KE_1$ and, in case $d$ is composite, $d \in \KE_d$.

\begin{theorem} \label{thm:cover}
The following statements hold:

\begin{enumerate}
\item We have $\SH_1 = \KK_1 = \CN$ and $\SH_d \subset \KE_d$ for $d \geq 2$.
\item All elements of $\SH_d$ have the same parity as $d$ for $d \geq 1$.
\item A cover of the set $\SH$ is
\[
  \SH = \bigcup_{d \, \geq \, 1} \SH_d.
\]
\end{enumerate}
\end{theorem}

\begin{proof}
We have to show three parts:

(i).~We have $\SH_1 = \KK_1 = \CN$ by definition. Fix $d \geq 2$.
If $m \in \SH_d$, then Theorem~\ref{thm:rho} implies that
$d \equiv \rho(m) \pmod{\lambda(m)}$.
By \eqref{eq:lambda-S-2} and \eqref{eq:rho-def} this translates to
$d \equiv m \pmod{\lambda(m)}$. Finally, \eqref{eq:lambda-congr} and
\eqref{eq:Kd-def} imply that $m \in \KE_d$.
This shows that $\SH_d \subseteq \KE_d$.

By~\eqref{eq:Kd-elem} there exists a prime $p > d$ such that
$m' = dp \in \KK_d \subseteq \KE_d$. Since $s_p(m') = d < p$,
it follows that $m' \notin \SH$.
This implies that $\SH_d \neq \KE_d$, and finally $\SH_d \subset \KE_d$.

(ii).~Fix $d \geq 1$ and $m \in \SH_d$. As in part~(i) we have
$d \equiv \rho(m) \equiv m \pmod{\lambda(m)}$.
By~\eqref{eq:rho-par} the result follows.

(iii).~Set $\mathcal{U} = \bigcup_{d \, \geq \, 1} \SH_d$. Since
$\SH_d \subset \SH$ for $d \geq 1$, it follows that $\mathcal{U} \subseteq \SH$.
By Theorem~\ref{thm:rho} we obtain for any $m \in \SH$ an index $d = \rho(m)$
such that $m \in \SH_d$. As a consequence, $\SH \subseteq \mathcal{U}$ and
finally $\SH = \mathcal{U}$.
\end{proof}

%%%%%%%%%%%%%%%%%%%%%%%%%%%%%%%%%%%%%%%%%%%%%%%%%%%%%%%%%%%%%%%%%%%%%%%%%%%%%%%%
% Bibliography
%%%%%%%%%%%%%%%%%%%%%%%%%%%%%%%%%%%%%%%%%%%%%%%%%%%%%%%%%%%%%%%%%%%%%%%%%%%%%%%%


\begin{thebibliography}{10}
\setlength{\itemsep}{3pt}

\bibitem{AGP:1994}
W.~R.~Alford, A.~Granville, and C.~Pomerance,
\newblock \emph{There are infinitely many Carmichael numbers},
\newblock Ann. of Math. \textbf{139} (1994), 703--722.

\bibitem{Beiler:1966}
A.~H.~Beiler,
\newblock \emph{Recreations in the Theory of Numbers},
\newblock Dover, New York, 1966.

\bibitem{Carmichael:1910}
R.~D.~Carmichael,
\newblock \emph{Note on a new number theory function},
\newblock Bull. Amer. Math. Soc. \textbf{16} (1910), 232--238.

\bibitem{Carmichael:1912}
R.~D.~Carmichael,
\newblock \emph{On composite numbers $P$ which satisfy the Fermat congruence $a^{P-1} \equiv 1 \ (\bmod\ P)$},
\newblock Amer. Math. Monthly \textbf{19} (1912), 22--27.

\bibitem{Clausen:1840}
T.~Clausen,
\newblock \emph{Lehrsatz aus einer Abhandlung \"uber die Bernoullischen Zahlen},
\newblock Astr. Nachr. \textbf{17} (1840), 351--352.

\bibitem{Conrad:2016}
K.~Conrad,
\newblock \emph{Carmichael numbers and Korselt's criterion},
\newblock expository paper (2016), 1--3.
\newblock Link: \href{http://www.math.uconn.edu/~kconrad/blurbs/ugradnumthy/carmichaelkorselt.pdf}{carmichaelkorselt.pdf}.

\bibitem{Conway&Guy:1996}
J.~H.~Conway and R.~K.~Guy,
\newblock \emph{The Book of Numbers},
\newblock Springer--Verlag, New York, 1996.

\bibitem{Crandall&Pomerance:2005}
R.~Crandall and C.~B.~Pomerance,
\newblock \emph{Prime Numbers: A Computational Perspective}, 2nd ed.,
\newblock Springer, New York, 2005.

\bibitem{Erdos:1956}
P.~Erd\H{o}s,
\newblock \emph{On pseudoprimes and Carmichael numbers},
\newblock Publ. Math. Debrecen \textbf{4} (1956), 201--206.

\bibitem{Granville&Pomerance:2002}
A.~Granville and C.~Pomerance,
\newblock \emph{Two contradictory conjectures concerning Carmichael numbers},
\newblock Math. Comp. \textbf{71} (2002), 883--908.

\bibitem{Guy:2004}
R.~K.~Guy,
\newblock \emph{Unsolved Problems in Number Theory}, 3rd ed.,
\newblock Springer, New York, 2004.

\bibitem{Hardy:1940}
G.~H.~Hardy,
\newblock \emph{Ramanujan},
\newblock Cambridge Univ. Press, New York, 1940.

\bibitem{Harman:2008}
G.~Harman,
\newblock \emph{Watt's mean value theorem and Carmichael numbers},
\newblock Int. J. Number Theory \textbf{4} (2008), 241--248.

\bibitem{Heath-Brown:2007}
D.~R.~Heath-Brown,
\newblock \emph{Carmichael numbers with three prime factors},
\newblock Hardy--Ramanujan J. \textbf{30} (2007), 6--12.

\bibitem{Kellner:2017}
B.~C.~Kellner,
\newblock \emph{On a product of certain primes},
\newblock J. Number Theory \textbf{179} (2017), 126--141.

\bibitem{Kellner&Sondow:2017}
B.~C.~Kellner and J.~Sondow,
\newblock \emph{Power-sum denominators},
\newblock Amer. Math. Monthly \textbf{124} (2017), 695--709.

\bibitem{Kellner&Sondow:2018}
B.~C.~Kellner and J.~Sondow,
\newblock \emph{The denominators of power sums of arithmetic progressions},
\newblock Integers \textbf{18} (2018), Article A95, 1--17.

\bibitem{Knoedel:1953}
W.~Kn\"odel,
\newblock \emph{Carmichaelsche Zahlen},
\newblock Math. Nachr. \textbf{9} (1953), 343--350.

\bibitem{Knoedel:1953b}
W.~Kn\"odel,
\newblock \emph{Eine obere Schranke f\"ur die Anzahl der Carmichaelschen Zahlen kleiner als $x$},
\newblock Arch. Math. \textbf{4} (1953), 282--284.

\bibitem{Korselt:1899}
A.~Korselt,
\newblock \emph{Probl\`{e}me chinois},
\newblock L'Interm\'{e}diaire Math. \textbf{6} (1899), 142--143.

\bibitem{Makowski:1962}
A.~Makowski,
\newblock \emph{Generalization of Morrow's $D$ numbers},
\newblock Simon Stevin \textbf{36} (1962), 71.

\bibitem{Pinch:2007}
R.~G.~E.~Pinch,
\newblock \emph{The Carmichael numbers up to $10^{21}$},
\newblock Proceedings of Conference on Algorithmic Number Theory 2007,
\newblock A.~Ernvall-Hyt\"onen et al., eds.,
\newblock TUCS General Publication \textbf{46},
\newblock Turku Centre for Computer Science, 2007, 129--131.

\bibitem{PSW:1980}
C.~Pomerance, J.~L.~Selfridge, and S.~Wagstaff,
\newblock \emph{The pseudoprimes to $25 \cdot 10^9$},
\newblock Math. Comp. \textbf{35} (1980), 1003--1026.

\bibitem{Ribenboim:2012}
P.~Ribenboim,
\newblock \emph{The New Book of Prime Number Records},
\newblock Springer, New York, 2012.

\bibitem{Robert:2000}
A.~M.~Robert,
\newblock \emph{A Course in $p$-adic Analysis}, GTM \textbf{198},
\newblock Springer--Verlag, New York, 2000.

\bibitem{Staudt:1840}
K.~G.~C.~von~Staudt,
\newblock \emph{Beweis eines Lehrsatzes die Bernoullischen Zahlen betreffend},
\newblock J. Reine Angew. Math. \textbf{21} (1840), 372--374.

\end{thebibliography}
\end{document}